\documentclass{amsart}
\usepackage{amssymb,amsmath, amsthm,latexsym}
\usepackage[toc,page]{appendix} 
\usepackage{graphics}
\usepackage{psfrag}
\usepackage{amscd}
\usepackage{graphicx}
\usepackage{here}
\newcommand{\cal}[1]{\mathcal{#1}}
\theoremstyle{plain}
\newtheorem{theorem}{Theorem}

\newtheorem{lemma}{Lemma}[section]
\newtheorem{theo}[lemma]{Theorem}
\newtheorem{proposition}[lemma]{Proposition}
\newtheorem{corollary}[lemma]{Corollary}
\theoremstyle{definition}
\newtheorem*{defi}{Definition}

\newtheorem{remark}[lemma]{Remark}

\parskip=\bigskipamount

\let\egthree=\phi
\let\phi=\varphi
\let\varphi=\egthree

\begin{document}
\title[Random walks ]
{On random walks on the mapping class group}
\author{Ursula Hamenst\"adt}
\date{July 13, 2025}

\begin{abstract}
We define an electrification of the curve
graph of a surface $S$ of finite type and use it to identify the Poisson boundary of a random walk
on the mapping class group of $S$ with some logarithmic moment condition as
a stationary measure on the space of minimal and maximal geodesic 
laminations on $S$, equipped with the Hausdorff topology.
\end{abstract}

\maketitle

\section{Introduction}

The \emph{mapping class group} ${\cal M\cal C\cal G}$ of
a surface of finite type, that is, a
closed surface $S$ of genus $g\geq 0$ from which $m\geq 0$
points, so-called \emph{punctures},
have been deleted, is the group of isotopy classes of diffeomorphisms 
of $S$. 
We assume that $S$ is \emph{non-exceptional},
that is, $S$ is not a sphere with at most $4$
punctures or a torus with at most $1$ puncture.
The mapping class group is finitely presented.

A \emph{random walk} on the mapping class groups is generated by a probability
measure $\mu$ on ${\cal M\cal C\cal G}$. The $n$-th convolution of $\mu$ is the
measure defined inductively by $\mu^{(n)}=\sum_h \mu(h)\mu^{(n-1)}(h^{-1}g)$, with
$\mu^{(1)}=\mu$. We then can consider
the \emph{Poisson boundary} of the random
walk which is a maximal \emph{stationary measure} $\nu$. Here a measure $\nu$ on
some measure space with a ${\cal M\cal C\cal G}$-action is called
stationary if it satisfies the identity
\[\nu=\sum_h \mu(h)h_*\nu=\nu.\]

The Poisson boundary is a purely measure theoretic object, but we can ask whether it
can be realized on a topological ${\cal M\cal C\cal G}$-space $X$. By this we mean
that there exists a probability measure $\nu$ on $X$ whose measure class
is invariant and so that 
for the ${\cal M \cal C\cal G}$-action on $X$, the measure $\nu$ can be identified with 
the Poisson boundary, defined on the Borel $\sigma$-algebra on $X$. 

Kaimanovich and Masur \cite{KM96} discovered that 
the Poisson boundary of ${\cal M\cal C\cal G}$ can be realized on the
space ${\cal P\cal M\cal L}$ of \emph{projective measured geodesic laminations} on $S$
provided that the measure $\mu$ has finite entropy and finite first logarithmic
moment with respect to the \emph{Teichm\"uller distance} $d_{\cal T}$ on
\emph{Teichm\"uller space} ${\cal T}(S)$. This means that for any point
$x\in {\cal T}(S)$, we have
$\sum_h \mu(h)(\max\{0,\log d_{\cal T}(x,hx)\})<\infty$.

Gadre and Maher \cite{GM18} showed that under the stronger finite moment
condition $\sum_h \mu(h)d(e,h)<\infty$ where $d$ is a word metric on
${\cal M\cal C\cal G}$, almost every biinfinite path converges to a point
in ${\cal P\cal M\cal L}$, and the pair of forward and backward limit points
define a Teichm\"uller geodesic in the \emph{principal stratum of marked quadratic differentials}, 
consisting of differentials with only simple zeros. 
Moreover, the same holds true if the support of $\mu$ not necessarily generates
${\rm Mod}(S)$ but generates a semi-group $H$ that
is non-elementary and contains a
pseudo-Anosov mapping class belonging to the principal stratum. By this we mean
that a quadratic differential tangent to its invariant Teichm\"uller geodesic is
contained in the principal stratum.

The main purpose of this note is to provide a different perspective on this result
and show the following strengthening. 
Recall that the \emph{curve graph} ${\cal C\cal G}(S)$ of $S$ is the 
graph whose vertices are isotopy classes of essential simple closed curves on 
$S$ and where two such curves are connected by an edge of length one if they
can be realized disjointly. The curve graph ${\cal C\cal G}(S)$, equipped
with the natural simplicial metric $d_{{\cal C\cal G}(S)}$, is a 
Gromov hyperbolic geodesic metric graph \cite{MM99}, and the mapping class group acts on 
${\cal C\cal G}(S)$ as a group of simplicial automorphisms. 

The Gromov boundary $\partial {\cal C\cal G}(S)$ of the curve graph ${\cal C\cal G}(S)$ 
can naturally be identified with the space of \emph{minimal filling geodesic laminations}, 
equipped with the so-called \emph{coarse Hausdorff topology}. Here a geodesic lamination is 
filling if it decomposes $S$ into ideal polygons and once punctured ideal polygons. If these
polygons are all
ideal triangles and once punctured monogons, then the geodesic lamination is called 
\emph{maximal}.

A probability measure $\mu$ on ${\cal M\cal C\cal G}$ has 
\emph{finite logarithmic moment} for the action on ${\cal C\cal G}(S)$ if 
$\sum_h \mu(h) (\max\{ d_{{\cal C\cal G}(S)}(c,hc)),0\})<\infty$.  
Note that this condition is weaker than stating that the logarthmic moment for the
action on ${\cal T}(S)$ is finite. 

\begin{theorem}\label{poisson}
Let $\mu$ be a probability measure on ${\cal M\cal C\cal G}$  whose support
  generates a non-elementary semi-subgroup of ${\cal M\cal C\cal G}$ which contains at
least one pseudo-Anosov element belonging to the principal stratum.
Assume moreover that $\mu$ has finite entropy and finite
  logarithmic moment for the action on the curve graph. Then 
the Poisson boundary of the random walk on ${\cal M\cal C\cal G}$
  generated by $\mu$ can be realized
 as a stationary measure on the space of minimal and maximal geodesic
  laminations on $S$, equipped with the Hausdorff topology. 
\end{theorem}
  
The proof of Theorem \ref{poisson} uses a result of independent interest which we explain next.

\begin{defi}\label{princ}
  The \emph{principal curve graph} of a closed surface $S$ of genus $g\geq 2$
  is the graph whose vertices are simple closed curves and where
  two such curves $c,d$ are connected by an edge of length one if
  $S\setminus(c\cup d)$ has a complementary component which neither is 
  a fourgon nor a sixgon nor a once punctured bigon.
\end{defi}  

We show in Section \ref{theprincipal} and Section \ref{gromov}. 

\begin{theorem}\label{hyp}
  The principal curve graph is a quasi-tree of infinite diameter.
 Its Gromov boundary
  equals the space of minimal and maximal geodesic laminations, equipped
  with the Hausdorff topology.
\end{theorem}

Theorem \ref{poisson} is obtained from Theorem \ref{hyp} and the work of Maher and Tiozzo \cite{MT18}.  
The principal curve graph and its variants mentioned in Remark \ref{intermediate}
is also very useful for the understanding of the geometry of the 
curve graph and the mapping class group, but we defer this discussion to
forthcoming work.

The main part of the 
article is divided into three sections. In Section \ref{theprincipal} we introduce the principal
curve graph and show that it is hyperbolic, of infinite diameter. In Section \ref{gromov} we show that
it is a quasi-tree whose Gromov boundary can be identified with the space of minimal and maximal
geodesic laminations, equipped with the Hausdorff topology. 
In Section \ref{random} we prove Theorem \ref{poisson}.

\noindent{\bf Acknowledgement:} This work was carried out when the author visited the MSRI in Berkeley
in fall 2015.

\section{The principal curve graph}\label{theprincipal}

The vertices of the \emph{curve graph} ${\cal C\cal G}(S)$ of 
a non-exceptional 
surface $S$ of finite type are isotopy classes of
essential simple closed curves on $S$, 
that is, simple closed curves which are not contractible or homotopic into
a puncture, and where two such vertices
are connected by an edge of length one if they can be realized disjointly.

Let $c,d$ be two simple closed curves on $S$ in minimal position, that is,
$c,d$ intersect in the minimal number of points. In the sequel we always assume
that this is the case.  
Then $S\setminus (c\cup d)$ is
a union of complementary regions whose boundaries consist of
subarcs of $c$ and $d$ in alternating order. In particular, a
complementary component which is simply connected is 
a polygon with an even number of sides, and a complementary
component which is a punctured disc is a punctured polygon
with an even number of sides.

The curves $c,d$ \emph{bind} $S$ if each component of 
$S\setminus (c\cup d)$ is 
a disc or a once punctured disc. This is equivalent to stating 
that there is no
component of $S\setminus (c\cup d)$ 
which contains an essential simple closed curve. 

Let us assume in the sequel that $c,d$ bind $S$. 
By reasons of Euler
characteristic, there is at least one 
component of $S\setminus (c\cup d)$ which is a polygon 
with at least 6 sides or a once punctured polygon with 
at least 4 sides. As in the introduction, we define the 
principal curve graph ${\cal P\cal C}(S)$
of $S$  as the graph whose vertices
  are isotopy classes of essential
  simple closed curves on $S$ 
  and where two such curves $c,d$ are connected
  by an edge of length one if there is a component of 
  $S\setminus (c\cup d)$ 
which either contains an essential simple closed curve of $S$, or is 
a polygon with at least 8 sides or a once punctured polygon with at least 
6 sides. Clearly the mapping class group acts on 
${\cal P\cal C}(S)$ as a group of simplicial automorphisms.

The goal of this section is to show.

\begin{proposition}\label{hyperbolicity}
  The principal curve graph is a hyperbolic geodesic metric space of
infinite diameter.  
\end{proposition}

The principal curve graph ${\cal P\cal C}(S)$ 
contains a ${\cal M\cal C\cal G}$-invariant 
subgraph ${\cal C\cal G}_0(S)$, with the same set of vertices, that is, 
vertices are simple closed curves, and where two such vertices $c,d$ are
connected by an edge of length one if 
there is a component of $S\setminus (c\cup d)$ 
which contains an essential simple closed curve of $S$. 
Then this curve  is disjoint
from both $c,d$ and therefore 
the distance in the curve graph ${\cal C\cal G}(S)$ of $S$ 
between $c,d$ is at most two.
Vice versa, if $c,d$ are connected in ${\cal C\cal G}(S)$ by an edge,
then they are disjoint. If $S$ is not a five punctured sphere or a twice punctured torus, 
then $S\setminus (c\cup d)$ contains a component which is 
not simply connected and $c,d$ are connected in ${\cal C\cal G}_0(S)$ by an edge.
If $S$ is a five punctured sphere or a twice punctured torus then it is easy to see that there
is a simple closed curve $e$ which is connected to both $c,d$ by an edge in ${\cal C\cal G}_0(S)$.
Thus we have shown.

\begin{lemma}\label{subgraph}
${\cal C\cal G}_0(S)$ is two-quasi isometric to the curve graph ${\cal C\cal G}(S)$ of $S$. 
\end{lemma}

As ${\cal P\cal C}(S)$ is obtained from ${\cal C\cal G}_0(S)$ by adding some edges, 
the graph 
${\cal P\cal C}(S)$ can be thought of as an electrification of ${\cal C\cal G}(S)$.  
If $d_{{\cal P\cal C}(S)}(c,d)\geq 2$ then we say
that $c,d$ \emph{completely fill} $S$.

A \emph{train track} on $S$ is an embedded
1-complex $\tau\subset S$ whose edges
(called \emph{branches}) are smooth arcs with
well-defined tangent vectors at the endpoints. At any vertex
(called a \emph{switch}) the incident edges are mutually tangent.
Through each switch there is a path of class $C^1$
which is embedded
in $\tau$ and contains the switch in its interior. In
particular, the half-branches which are incident
on a fixed switch are divided into two classes according to
the orientation of an inward
pointing tangent at the switch. Each closed curve component of
$\tau$ has a unique bivalent switch, and all other switches are at
least trivalent.
The complementary regions of the
train track have negative Euler characteristic, which means
that they are different from discs with $0,1$ or
$2$ cusps at the boundary and different from
annuli and once-punctured discs
with no cusps at the boundary.
A train track is called \emph{maximal} if each of
its complementary components either is a trigon,
that is, a topological disc with three cusps at the boundary,
or a once punctured monogon, that is, a once punctured disc
with one cusp at the boundary. 
We always identify train
tracks which are isotopic.
The book \cite{PH92} contains a comprehensive treatment
of train tracks which we refer to throughout the article.

A train track is called \emph{generic} if all switches are
at most trivalent.
The train track $\tau$ is called \emph{transversely recurrent} if
every branch $b$ of $\tau$ is intersected by an embedded simple
closed curve $c=c(b)\subset S$ of class $C^1$ 
which intersects $\tau$
transversely and is such that $S\setminus(\tau\cup c)$ does not contain an
embedded \emph{bigon}, that is, a disc with two corners at the
boundary.


A \emph{transverse measure} on a train track $\tau$ is a
nonnegative weight function $\mu$ on the branches of $\tau$
satisfying the \emph{switch condition}:
for every switch $s$ of $\tau$, 
the half-branches incident on $s$ are divided into two
classes, and the sums of the weights
over all half-branches in each of the two classes 
coincide.
The train track is called
\emph{recurrent} if it admits a transverse measure which is
positive on every branch. 
We denote by ${\cal T\cal T}$ the set of all recurrent and transversely recurrent
(marked) train tracks on $S$. The mapping class group of $S$ acts on 
${\cal T\cal T}$ as a group of transformations.

A \emph{geodesic lamination} for a complete
hyperbolic structure on $S$ of finite volume is
a \emph{compact} subset of $S$ which is foliated into simple
geodesics. A geodesic lamination $\lambda$ is \emph{minimal}
if each of its half-leaves is dense in $\lambda$.
A geodesic lamination is \emph{maximal}
if its complementary regions are all ideal triangles
or once punctured monogons (note that a minimal geodesic
lamination can also be maximal).
The space of geodesic laminations on $S$
equipped with the \emph{Hausdorff topology} is
a compact metrizable space.

A geodesic lamination $\lambda$
is called \emph{complete} if $\lambda$ is maximal and
can be approximated in the Hausdorff topology by
simple closed geodesics. The space ${\cal C\cal L}$
of all complete geodesic laminations equipped with
the Hausdorff topology is compact. The mapping
class group ${\cal M\cal C\cal G}(S)$ naturally
acts on ${\cal C\cal L}$ as a group of 
homeomorphisms.
Every geodesic lamination $\lambda$
which is a disjoint union of finitely many minimal components
is a \emph{sublamination} of
a complete geodesic lamination, that is, there
is a complete geodesic lamination which contains
$\lambda$ as a closed subset (Lemma 2.2 of \cite{H09}).

A train track or a geodesic lamination $\sigma$ is
\emph{carried} by a transversely recurrent train track $\tau$ if
there is a map $\phi:S\to S$ of class $C^1$ which is homotopic to the
identity and maps $\sigma$ into $\tau$ in such a way 
that the restriction of the differential of $\phi$
to the tangent space of $\sigma$ vanishes nowhere;
note that this makes sense since a train track has a tangent
line everywhere. We call the restriction of $\phi$ to
$\sigma$ a \emph{carrying map} for $\sigma$.
Every geodesic lamination $\lambda$ which is carried
by $\sigma$ is also carried by $\tau$. A train track
$\tau$ is called \emph{complete} if it is generic and
transversely recurrent and if it carries
a complete geodesic lamination.
The space
of complete geodesic laminations carried by a complete
train track $\tau$ is open and closed in ${\cal C\cal L}$
(Lemma 2.3 of \cite{H09}).
In particular, the space ${\cal C\cal L}$
is totally disconnected.

The following lemma is based on a construction 
due to Masur and Minsky (Section 4 of \cite{MM04}).

\begin{lemma}\label{oneswitch}
  Let $c,d$ be any two simple closed curves on $S$ which bind $S$.
  Then there exist a recurrent one-switch train track
  $\eta(c,d)$ which carries $d$ and intersects $c$ in a single point.
This train track is maximal only if $c,d$ completely fill $S$.
\end{lemma}  
\begin{proof} Using the notations from the lemma, 
choose a component $I$ of $c\setminus d$  
contained in the boundary of a polygonal component of 
$S\setminus (c\cup d)$  with at least 
$6$ sides or a once punctured polygonal
component of $S\setminus (c\cup d)$ with at least four sides. 
Such a component $C$ always exists by reasons of Euler characteristic.
If the distance between $c,d$ in ${\cal P\cal C}(S)$ equals one, then
we may assume that $C$ is a polygonal component with at least $8$ sides or a once
punctured polygonal component with at least $4$ sides.
Contract $c\setminus I$ to a point. The graph $G$ obtained from $c\cup d$ in this way
has a single vertex $p$ which is the image of $c\setminus I$. Furthermore, it 
intersects $c$ in a single point. 
Impose a switch structure
at the vertex $p$ as follows. 

Declare all half-edges of $G$ which are subarcs of $d$ and leave 
$c$ to a fixed side to be incoming, and declare the half-edges of $G$ which are
subarcs of $d$ and leave $c$ to the opposite side to be outgoing. 
The result of this construction is
a \emph{bigon track} $\hat \sigma$, that is, a graph which has all properties of a train track
except that it may contain bigons. These bigons can be 
collapsed to yield a train track $\sigma$ 
which carries $d$  and intersects $c$ in a single point. As $\sigma$ carries the simple closed
curve $d$ and is filled by $d$, that is, a carrying map $d\to \sigma$ to surjective, 
the train track $\sigma$ is recurrent.

Let us inspect the complementary regions of $\hat \sigma$. 
If $E$ is a component of $S\setminus (c\cup d)$ with $2\ell$ sides not containing $I$, 
then each side $e$ of $E$ contained in $c$ is contracted to a point
in the construction of $\hat \sigma$, and the two sides of $E$ adjacent to $e$ 
meet at a cusp of the complementary region of $\hat \sigma$ which is the collapse of 
$E$. Thus any complementary polygon (or once punctured complementary polygon) 
of $S\setminus (c\cup d)$ with $2\ell$ sides not containing $I$ 
gives rise to a complementary component of $\hat \sigma$ which is a topological disk 
(or once punctured disk) with $\ell$ cusps in the boundary. 
Complementary quadrangles collapse to complementary bigons.

The side $I$ of the component $C$ is removed, and the component $C$ merges with 
the second component $C^\prime$ of $S\setminus (c\cup d)$ which contains $I$ in its boundary
to a complementary component $D$ of $\hat \sigma$. 
To analyze this component, we distinguish three cases. 

 If $C^\prime=C$ then $D$ contains an essential
 simple closed curve.  Namely, in this case 
 there is a simple closed curve in $C\cup I$ which intersects $c$ in a single point 
 contained in $I$, and this simple closed curve is contained in the complementary
 region $D$ of $\hat \sigma$. 

Now let us assume that $C$ is a polygon without puncture
and $2\ell\geq 6$ sides and that $C^\prime\not=C$. 
If $C^\prime$ is a polygon 
 with $2k\geq 4$ sides, then $D$ is disk with 
$\ell+k-2\geq \ell$ cusps in the boundary. 
 Similarly, if $C^\prime$ is a once punctured polygon with 
 $2k\geq 2$ sides, then $D$ is a once punctured disk with $\ell+k-2\geq \ell-1$ 
 cusps in the boundary.

If both $C,C^\prime$ are once punctured polygons, 
then the component $D$ contains two punctures and hence it contains
an essential simple closed curve on $S$ surrounding the punctures.

 As a consequence, if $C$ is a polygon with at least $8$ sides
 or a once punctured polygon with at least $4$ sides, then 
 the bigon track $\hat \sigma$ and hence the train track $\sigma$ contains a complementary
 component which either contains an essential simple closed curve, or is a polygon with at least
 $8$ sides, or is a once punctured polygon with at least $4$ sides.
 In particular, $\sigma$ is not maximal. 
This completes the proof of the lemma.
\end{proof}  

Denote by 
\[\Psi:{\cal C\cal G}(S)\to {\cal P\cal C}(S)\] the map induced by the vertex inclusion.
This map is one-Lipschitz if $S$ is distinct from a five punctured sphere or a twice punctured torus
and is two-Lipschitz otherwise. 
An \emph{unparameterized $L$-quasi-geodesic} in a geodesic metric space $X$ is a map
$\psi:[a,b]\to X$ with the property that there exists an increasing homeomorphism
$\rho:[0,c]\to [a,b]$ such that $\psi \circ \rho:[0,c]\to X$ is an $L$-quasi-geodesic.

A \emph{vertex cycle} of a recurrent train track $\tau$ is
an immersed simple closed curve in $\tau$ which is an extreme
point for the cone of all transverse measures for $\tau$.
Such a vertex cycle either is an embedded simple closed curve in
$\tau$ or a dumbbell. In particular, a closed trainpath defined
by a vertex cycle passes through any branch at most twice.
As a consequence, the geometric intersection number between any
two vertex cycles on $\tau$ is uniformly bounded, and 
there is a coarsely well defined map
\begin{equation}\label{upsilon}
  \Upsilon:{\cal T\cal T}\to {\cal C\cal G}(S)\end{equation}
which associates to a train track one of its vertex cycles. 
Here coarsely well defined means that $\Upsilon$ depends on choices,
but any two choices give rise to maps which map a given point to
images of uniformly bounded distance.
We refer to \cite{H06} for more information on this construction.  We use Lemma \ref{oneswitch}
to show.

\begin{proposition}\label{ccgraphqtree}
The principal curve graph ${\cal P\cal C}(S)$ is hyperbolic. Geodesics
in ${\cal C\cal G}(S)$ map by $\Psi$  
to uniform unparameterized quasi-geodesics in ${\cal P\cal C}(S)$. 
\end{proposition}
\begin{proof} The proposition follows from 
the work of Kapovich and Rafi \cite{KR14} if
we can verify that the conditions in Corollary 2.4 of
\cite{KR14} are fulfilled. For this it suffices to show 
the existence of a number $L>1$ with the following property. 
Whenever $c,d$ are curves whose distance in the principal curve
graph is one, then there exists an $L$-quasi-geodesic 
$\rho:[0,m]\to {\cal C\cal G}(S)$ connecting $\rho(0)=c$ to 
$\rho(m)=d$ 
such that the diameter of 
$\Psi(\rho[0,m])\subset {\cal P\cal C}(S)$ is at most $L$.

This is obvious if
$c, d$ do not bind $S$ as in this case, their distance
in the curve graph is at most two. 
If $c, d$ bind $S$ then 
we construct such a quasi-geodesic using the fact that so-called 
\emph{splitting sequences} of
train tracks on $S$ 
are mapped by the map $\Upsilon$ to uniform unparameterized 
quasi-geodesics in ${\cal C\cal G}(S)$ \cite{H06}. 

Let $\sigma$ be a one-switch train track constructed from $c,d$ 
as in Lemma \ref{oneswitch}.
Then $\sigma$ has a complementary component different from a
trigon or a once punctured monogon. 
In particular, any two simple closed curves
carried by $\sigma$ have distance at most  one in ${\cal P\cal C}(S)$. 
Furthermore, the distance in ${\cal P\cal C}(S)$ 
between a vertex cycle of $\sigma$ and the curve $c$ 
is uniformly bounded since this holds true for their distance
in the curve graph.

By Theorem 2.4.1 of \cite{PH92}, 
there is a \emph{splitting and collision} sequence $\sigma_i$ 
issuing from $\sigma$ 
which consists of train tracks which carry $d$ and
which connects $\sigma$ to 
the train track which consists of the single curve $d$. 
Each of the train tracks $\sigma_i$ is carried by $\sigma$. 
Associating to each train track $\sigma_i$ one of its
vertex cycles $c_i$ 
defines a uniform unparameterized quasi-geodesic in 
${\cal C\cal G}(S)$ \cite{H06} 
which connects a vertex cycle $c_0$ of 
$\sigma=\sigma_0$, that is, a 
simple closed curve in a uniformly bounded
neighborhood of $c$ in ${\cal C\cal G}(S)$, 
to the curve $d$. Since each of the curves $c_i$ is carried 
by $\sigma$,
this quasi-geodesic consists of curves whose
distance to $d$ in the principal curve graph equals at most one.

Hyperbolicity of the principal curve graph now follows
from Corollary 2.4 of \cite{KR14}. 
This result also shows that geodesics in ${\cal C\cal G}(S)$
map to uniform 
unparameterized quasi-geodesic
in ${\cal P\cal C}(S)$.
This is what we wanted to show.
 \end{proof}

 Proposition \ref{ccgraphqtree} does not imply that the
 principal curve graph has infinite diameter. We next show that
 this is indeed the case. 
 
 A \emph{measured geodesic lamination} is a geodesic lamination
equipped with a transverse invariant measure. A projective
measured geodesic lamination is an equivalence class of measured
laminations which are obtained from each other by scaling.
The space ${\cal P\cal M\cal L}$ of projective measured geodesic
laminations, equipped with the weak$^*$-topology, is compact. 
A simple closed curve admits a unique transverse
measure up to scale and hence 
can be viewed as a projective
measured geodesic lamination.

By a construction due to Thurston and Veech, simple closed 
 curves $c,d$ which bind $S$ determine a line of 
 \emph{area one
 holomorphic quadratic differentials} on
 $S$ with horizontal and 
 vertical measured geodesic lamination
 supported in $c,d$, respectively. These differentials define
the cotangent line of a \emph{Teichm\"uller geodesic}. If $c,d$ 
are connected by an edge in the principal curve graph, then 
these quadratic differentials either have a zero of order at least two, or one of the 
marked points (punctures) is a regular point or a zero of the differential
and not a simple pole. We refer to \cite{H24} for a more comprehensive discussion.

 Using compactness of the space
${\cal P\cal M\cal L}$ 
of projective measured geodesic laminations we observe.

\begin{lemma}\label{converge}
  Let $(c_i,d_i)$ be a sequence of pairs
  of simple closed curves
such that $d_{{\cal P\cal C}(S)}(c_i,d_i)\leq 1$ for all $i$. 
Assume that the 
sequence $c_i$ converges as $i\to \infty$ in ${\cal P\cal M\cal L}$ 
to a projective measured geodesic lamination 
whose support $\mu$ is both minimal and maximal. 
Then up to passing to a subsequence,
the sequence $d_i$ converges to a projective 
measured geodesic lamination with support $\mu$.
\end{lemma}
\begin{proof} Let $\rho\in {\cal P\cal M\cal L}$ be the limit of 
the sequence $c_i$. The support of $\rho$ equals $\mu$. 
Using the notations from the lemma,  
by compactness and passing to 
a subsequence of the sequence $d_i$, we may assume that
$d_i$ converges in ${\cal P\cal M\cal L}$ to
a projective measured geodesic lamination $\xi$.

We argue by contradiction and assume that 
the support of $\xi$ is distinct from $\mu$.
Since $\mu$ is minimal and maximal, this implies that
$\xi$ together with the projective measured geodesic lamination
$\rho$ determines a Teichm\"uller geodesic
whose cotangent line $\gamma$
consists of area one quadratic differentials with 
vertical and horizontal projective measured geodesic laminations
$\rho,\xi$, respectively.

Namely, two projective measured geodesic laminations $\rho,\xi$ 
determine a
Teichm\"uller geodesic if and only if for every
measured geodesic lamination $\zeta$, the 
intersection between $\zeta$ and one of the two laminations 
$\rho,\xi$ is positive. 
Note that this property is invariant under scaling
and hence makes sense for
projective measured geodesic laminations.
As this is an open condition (for example, by continuity of the intersection form), we conclude
that for sufficiently large $i$ the pair 
$(c_i,d_i)$ binds $S$ and hence 
defines the cotangent line $\gamma_i$ 
of a Teichm\"uller geodesic, that is, an orbit of the \emph{Teichm\"uller flow}
on the bundle of area one quadratic differentials. 
Furthermore, by continuity, the cotangent lines $\gamma_i$  
of these Teichm\"uller geodesics  converge locally
uniformly to $\gamma$ in the bundle over
Teichm\"uller space whose fiber over a Riemann surface $X$ is the
sphere of area one quadratic differentials on $X$.

On the other hand, as the distance in the principal curve
graph between $c_i,d_i$ equals one, the cotangent line $\gamma_i$ 
is \emph{not} contained in the \emph{principal
stratum} of quadratic differentials with only simple zeros
and simple poles at the marked points 
(here we view an abelian differential as a quadratic differential
with all zeros of even order). 
As the complement of the principal stratum in the
Teichm\"uller space of quadratic differentials is closed 
and invariant under the Teichm\"uller flow,
the cotangent line of the limiting Teichm\"uller geodesic
is contained in the complement of
the principal stratum as well. But 
any quadratic differential whose
horizontal measured lamination is supported in a minimal 
complete geodesic lamination is contained in the principal stratum. 
This is a contradiction which shows the lemma.
\end{proof}

We are now ready to show that the diameter
of ${\cal P\cal C}(S)$ is infinite and complete the proof of 
Proposition \ref{hyperbolicity}. 
 %

\begin{proposition}\label{infinitediam}
The diameter of ${\cal P\cal C}(S)$ is infinite.
\end{proposition}
\begin{proof} The proof of this proposition is a variation of
  an argument of Luo as explained in Section 4.3 of  \cite{MM99}. 

Let $\lambda$ be a minimal complete geodesic lamination.
Then $\lambda$ determines a point in the Gromov boundary 
$\partial {\cal C\cal G}(S)$ of the curve graph ${\cal C\cal G}(S)$.
There exists $p>1$ and an edge path  
$\gamma:[0,\infty)\to {\cal C\cal G}(S)$ which is a 
$p$-quasi-geodesic and connects
a fixed simple closed curve $c_0$ to $\lambda$.

By Proposition \ref{ccgraphqtree}, the assignment $i\to c_i=\gamma(i)$
defines a uniform unparameterized quasi-geodesic in the
principal curve graph. 
Let us assume to the contrary that this quasi-geodesic is of finite
diameter, that is, that 
$d_{{\cal P\cal C}(S)}(c_0,c_i)$ is uniformly bounded.
Then by passing to a subsequence, we may assume that there is
a number $k>0$ so that $d_{{\cal P\cal C}(S)}(c_0,c_i)=k$ for all $i$.
Let $c_i^1\in {\cal P\cal C}(S)$ be a vertex so that 
$d_{{\cal P\cal C}(S)}(c_0,c_i^1)=k-1$  and
$d_{{\cal P\cal C}(S)}(c_i^1,c_i)=1$
for all sufficiently large $i$.

By Lemma \ref{converge},
we know that $c_i^1\to \lambda$ in the Hausdorff topology.
Repeat this argument with the sequence $c_i^1$.
After $k$ such steps we conclude that $c_0\to \lambda$ in the 
Hausdorff topology, which is a contradiction.
\end{proof}

\section{Gromov boundary and the action
  of the mapping class group} \label{gromov}

Since the graph ${\cal P\cal C}(S)$ is hyperbolic of infinite
diameter, its \emph{Gromov boundary} $\partial {\cal P\cal C}(S)$ 
is nontrivial. Furthermore, since ${\cal M\cal C\cal G}$ acts on 
${\cal P\cal C}(S)$ as a group of simplicial isometries, it acts on 
$\partial {\cal P\cal C}(S)$ as a group of
transformations. 

\begin{proposition}\label{gromovboundary}
  The Gromov boundary of the principal curve graph is the
  space of minimal complete geodesic laminations, equipped with the
 Hausdorff topology.
\end{proposition}  
\begin{proof}
  A point in the Gromov boundary $\partial {\cal C\cal G}(S)$ of
  ${\cal C\cal G}(S)$ can be viewed as an equivalence class of
  uniform quasi-geodesic rays in ${\cal C\cal G}(S)$ where
  two such quasi-geodesic rays are equivalent if their
  Hausdorff distance in ${\cal C\cal G}(S)$ 
  in finite (this is true in the situation
  at hand in spite of the fact that the curve graph is not
  locally finite). 

  Since ${\cal C\cal G}(S)$ and ${\cal P\cal C}(S)$ have the
  same vertices and, by Proposition \ref{ccgraphqtree}, 
  any two points in ${\cal P\cal C}(S)$ can be connected by a uniform
  quasi-geodesic which is a reparameterization of a uniform quasi-geodesic
  in ${\cal C\cal G}(S)$ with the same endpoints, 
   we conclude that
  the Gromov boundary of the principal curve graph is
  a quotient of the subspace of the Gromov boundary
  of ${\cal C\cal G}(S)$ consisting of equivalence classes of
  those quasi-geodesic rays in ${\cal C\cal G}(S)$ 
  whose diameter in ${\cal P\cal C}(S)$
  are infinite.
  The quotient is taken by the equivalence relation
  which identifies two points if they correspond to quasi-geodesic rays
  whose Hausdorff distance in ${\cal P\cal C}(S)$ is finite.

  By Proposition \ref{infinitediam} and its proof, a uniform quasigeodesic 
  in ${\cal C\cal G}(S)$ connecting a basepoint $c$ to 
  a minimal complete geodesic lamination $\lambda$ is an unparameterized
  uniform quasi-geodesic in ${\cal P\cal C}(S)$ of infinite diameter.
  Thus $\lambda$ defines a point in the Gromov boundary of ${\cal P\cal C}(S)$.

We claim that
  two distinct such minimal complete geodesic
laminations define non-equivalent quasi-geodesic rays in
${\cal P\cal C}(S)$ and hence distinct points in the
Gromov boundary of ${\cal P\cal C}(S)$. Namely, such a pair of distinct
points can be connected by 
a biinfinite uniform quasi-geodesic in the
curve graph and hence by Proposition \ref{ccgraphqtree}, by  
a biinfinite unparameterized
quasi-geodesic in ${\cal P\cal C}(S)$.
By Proposition \ref{infinitediam}, its 
two half-rays have infinite diameter in ${\cal P\cal C}(S)$, which
is only possible if these laminations defined disctinct points
in the boundary of ${\cal P\cal C}(S)$.
As a consequence, the subspace of $\partial {\cal C\cal G}(S)$
of all minimal complete geodesic laminations embeds (as a set) 
into the Gromov boundary of ${\cal P\cal C}(S)$.

We show next that a minimal filling geodesic lamination which
is not maximal defines an equivalence class of uniform
quasi-geodesic rays in ${\cal C\cal G}(S)$ which have finite
diameter in ${\cal P\cal C}(S)$.

Thus 
let $\lambda$ be a minimal geodesic lamination which fills $S$ but
which is not maximal, that is, which is not complete.
By the results in Section 3 of \cite{H24}, there exists
a recurrent train track $\xi$ of the same \emph{combinatorial type} as $\lambda$ 
(that is, with complementary regions whose topological types coincide with the 
topological types of the complementary regions of $\lambda$)  
which carries $\lambda$. 
  In particular, at least one
  of the complementary components of $\xi$ is a polygon with more
  than three sides or a once punctured polygon with
  at least two sides.
   
  Let $(\xi_i)$ be a $\lambda$-splitting sequence beginning at
  $\xi_0=\xi$, that is, a splitting sequence so that each $\xi_i$ carries $\lambda$.
  For each $i$ let $c_i$ be a vertex
  cycle of $\xi_i$. As both $c_0,c_i$ are carried by
  $\xi$, there is at least one component of
  $S-(c_0\cup c_i)$ which is a polygon with at least
  $8$ sides or a once punctured polygon with at least four sides.
  This yields that the distance in
  ${\cal P\cal C}(S)$ between $c_0$ and $c_i$ is at most one 
  for all $i$.
  Since $i\to c_i$ is an unparameterized quasi-geodesic in the curve
  graph of $S$ which converges in 
  ${\cal C\cal G}(S)\cup \partial {\cal C\cal G}(S)$ to 
  $\lambda\in \partial {\cal C\cal G}(S)$  \cite{H06}, 
  this implies that $\lambda$ does not define a point 
  in the Gromov
  boundary of ${\cal P\cal C}(S)$.

  To summarize, as a set, the Gromov boundary of ${\cal P\cal C}(S)$ can be identified
  with the subspace $\partial {\cal P\cal C}(S)$ of
  $\partial {\cal C\cal G}(S)$ of all 
  minimal complete geodesic laminations. That this inclusion
  is a homeomorphism onto its image can be seen with
  the same arguments used before. Namely, as the inclusion
  ${\cal C\cal G}(S)\to {\cal P\cal C}(S)$
  is one-Lipschitz, the topology on $\partial {\cal P\cal C}(S)$ 
 as a subspace of $\partial {\cal C\cal G}(S)$ is finer than 
  the topology on $\partial {\cal P\cal C}(S)$ 
  as the Gromov boundary of 
  ${\cal P\cal C}(S)$. Thus it suffices to show that the inclusion
  map $\partial {\cal P\cal C}(S)\to \partial{\cal C\cal G}(S)$ is continuous
  where $\partial {\cal P\cal C}(S)$ is equipped with the topology
  as the Gromov boundary of ${\cal P\cal C}(S)$, and this is equivalent
  to stating that if $\lambda_i\to \lambda$ in
  $\partial{\cal P\cal C}(S)$, then $\lambda_i$ converges to
  $\lambda$ in the Hausdorff topology.

 To establish this claim we argue by contradiction and we assume that
 there exists a sequence $\lambda_i\subset 
 \partial {\cal P\cal C}(S)$ which converges to $\lambda\in \partial {\cal P\cal C}(S)$
 but does not converge to $\lambda$ in $\partial {\cal C\cal G}(S)$.
 By compactness of the space of geodesic
  laminations equipped with the Hausdorff topology, 
  up to passing to a subsequence,
  we may assume that $\lambda_i\to \mu$ in the Hausdorff topology
  where $\mu\not=\lambda$. 
  Assume without loss
  of generality that $\lambda_i\not=\lambda$ for all $i$.
  Equip $\lambda_i,\lambda$ with projective
  transverse measures $\alpha_i,\beta$. Then for each $i$, 
  the pair $(\alpha_i,\beta)$ of projective measured geodesic
  laminations determines a Teichm\"uller geodesic $\gamma_i$.
  
  By passing to another
  subsequence, we may assume that the
  Teich\-m\"ul\-ler geo\-de\-sics $\gamma_i$ converge locally uniformly 
  to a Teichm\"uller geodesic $\gamma$ defined by a pair 
  $(\alpha,\beta)$ of projective
  measured geodesic laminations, 
  where $\alpha$ is supported in $\mu\not=\lambda$.
Now by \cite{MM99}, associating to a point $X$ 
on the Teichm\"uller geodesic $\gamma$ 
a simple closed curve in ${\cal C\cal G}(S)$ whose length for the 
hyperbolic metric $X$ is uniformly bounded defines a uniform unparameterized
quasi-geodesic in ${\cal C\cal G}(S)$. As the Teichm\"uller geodesics
$\gamma_i,\gamma$ all pass through a fixed compact subset of 
Teichm\"uller space, together with 
Proposition \ref{ccgraphqtree}, this implies that
  there are uniform quasigeodesics in ${\cal P\cal C}(S)$ connecting
  $\lambda$ to $\lambda_i$ which pass through a fixed 
  subset of ${\cal P\cal C}(S)$ of uniformly bounded diameter, 
  violating the assumption that
$\lambda_i\to \lambda$ in $\partial {\cal P\cal C}(S)$.  
This completes the proof of the proposition.  
\end{proof}

Since the space of complete geodesic laminations, equipped with the Hausdorff topology, 
is totally disconnected, we obtain.

\begin{corollary}\label{discon}
The Gromov boundary of ${\cal P\cal C}(S)$ is totally disconnected.
\end{corollary}

It seems to be well known that a Gromov hyperbolic space with totally disconnected boundary
and satisfying some additional assumptions like extendibility of uniform quasi-geodesics is 
a quasi-tree, that is, it is quasi-isometric to a tree. As we were not able to find 
references which can be applied to the principal curve graph we provided 
a complete proof of the
following statement which completes the proof of Theorem \ref{hyp}. 

\begin{proposition}\label{quasitree}
The principal curve graph is a quasi-tree.
\end{proposition}
\begin{proof}
By \cite{Man05}, we have to verify the so-called \emph{bottleneck property}:
If $\gamma:[0,k]\to {\cal P\cal C}(S)$ is a geodesic edge path of length $k\geq 0$, 
then any edge path
connecting $\gamma(0)$ to $\gamma(k)$ passes through a uniformly bounded neighborhood of 
$\gamma(k/2)=e$. 


To see that this holds true we first establish the following

\smallskip\noindent
{\bf Claim:} Let $c\not=d$ be simple closed curves and let 
$\tau$ be a complete train track which carries $c$ but does not carry $d$. 
Then 
any edge path in ${\cal P\cal C}(S)$ 
connecting $c$ to $d$ passes through the one-neighborhood of some vertex cycle of $\tau$.

\smallskip
\noindent
\emph{Proof of the claim:} We proceed by induction on the length of an edge path $\alpha$ in 
the curve graph ${\cal C\cal G}(S)$ connecting some 
simple closed curve $c$ to some simple closed curve $d$. That this is legitimate follows from the 
fact that ${\cal C\cal G}(S)$ and ${\cal P\cal C}(S)$ have the same set of vertices. 

Assume first that this length equals one. If $c$ does not completely fill $\tau$ then 
 the image of $c$ under a carrying map 
$c\to \tau$ is not surjective. Then this image is a proper subtrack $\sigma$ of 
$\tau$, in particular, $\sigma$ is not maximal. As a consequence,
the distance in ${\cal P\cal C}(S)$ between $c$ and a vertex cycle of 
$\sigma$, which also is a vertex cycle of $\tau$, is at most one. 
 On the other hand, if $c$ completely fills $\tau$, then as $d$ is disjoint from $c$, 
 it follows from Lemma 4.5 of \cite{MM99} that $d$ is carried by $\tau$, a contradiction. This shows the
 claim for curves $c,d$ of distance one in ${\cal C\cal G}(S)$. 

Suppose the claim is true for all edge paths $\alpha:[0,m]\to {\cal C\cal G}(S)$ 
of length at most $m-1$ for some $m\geq 2$. 
Let $\alpha:[0,m]\to {\cal C\cal G}(S)$ be an edge path of length $m$ and let 
$\tau$ be a complete train track carrying $\alpha(0)$ but not $\alpha(m)$.
As before, if $\alpha(0)$ does not 
completely fill the train track $\tau$, then the distance between $\alpha(0)$ and 
a vertex cycle of $\tau$ equals one and we are done. Otherwise 
$\alpha(1)$ is carried by $\tau$ by Lemma 4.5 of \cite{MM99}. We then 
can apply the induction hypothesis 
to the path $\alpha[1,m]$ 
to complete the proof of the claim.
\hfill $\blacksquare$

Let $\gamma, e$ be as in the first paragraph of this proof. Assume without loss of generality that 
$e$ is a vertex of ${\cal P\cal C}(S)$, that is, $e$ is
a simple closed curve on $S$. Choose a pants decomposition $P$ of $S$ containing $e$, and choose moreover 
a system of \emph{spanning curves} for $P$. 
These data define a \emph{marking} of $S$. The marking in turn determines
a finite collection of \emph{train tracks in standard form} for the marking. Each of these train tracks either 
carries $e$ or carries a curve which intersects $e$ in at most
two points and is of distance at most two to $e$ in the 
curve graph ${\cal C\cal G}(S)$. 
Any geodesic lamination $\lambda$ on $S$ is carried by one of these trains tracks,
and if  $\lambda$ is complete, then it is carried by a unique such train track \cite{PH92,H09}. 
Let $\tau_1,\tau_2$ be the two of these train tracks which carry $c,d$. 
If $\tau_1\not=\tau_2$ then a path $\gamma$ connecting $c$ to $d$ passes through the three-neighborhood
of $e$ by the above claim and the choice of $\tau_1,\tau_2$. 

In the case that $\tau_1=\tau_2=\tau$ we modify $\tau$ with a splitting sequence of maximal length
so that the resulting train track $\eta$ still carries both $c,d$. Then no nontrivial split of $\eta$ carries both 
$c,d$. We then can modify 
$\eta$ with a single split so that the split track $\eta^\prime$ carries $c$ but not $d$. 
As the image under $\Upsilon$ of a splitting
sequence is a uniform unparameterized quasi-geodesic in ${\cal C\cal G}(S)$ and
${\cal P\cal C}(S)$,  the image under the map $\Upsilon$ 
of a splitting sequence connecting $\tau$ to $\eta^\prime$ is a subarc of 
a uniform quasi-geodesic connecting $e$ to both $c,d$. Since 
${\cal P\cal C}(S)$ is hyperbolic and $e$ is contained in a geodesic 
connecting $c$ to $d$, this implies that
the distance in ${\cal P\cal C}(S)$ 
between $e$ and a vertex cycle of $\eta^\prime$ is uniformly bounded. 
Thus the proposition follows from the above claim, applied to $c,d$ and $\eta^\prime$. 
\end{proof}

Recall that any pseudo-Anosov mapping class
$\phi\in {\cal M\cal C\cal G}$ has a unique
\emph{axis} in Teichm\"uller space ${\cal T}(S)$ of $S$,
that is, an invariant Teich\"uller geodesic
on which $\phi$ acts as a nontrivial translation.

\begin{corollary}\label{hyperbolic}
  Let $\phi\in {\cal M\cal C\cal G}$ be a pseudo-Anosov mapping class.
  Then $\phi$ acts as a hyperbolic isometry on the
  principal curve graph if and only if 
  the cotangent line of its axis
  is contained in the principal stratum of quadratic differentials
  with only simple zeros.
\end{corollary}  
\begin{proof}
A point
  $X\in {\cal T}(S)$ is a marked complete finite volume hyperbolic metric on $S$.  
  Define $\Upsilon_0(X)\in {\cal C\cal G}(S)$
  to be a simple closed curve of uniformly bounded $X$-length. Such a curve exists by the 
  Gauss Bonnet theorem which shows that the area of a hyperbolic metric on $S$ is 
  a topological invariant and hence a metric disc of fixed radius has to contain an essential 
  curve as it can not lift to a diffeomorphic disc in the hyperbolic plane. 
  If $\gamma$ is any Teichm\"uller geodesic, then the assignment
  $t\to \Upsilon_0(\gamma(t))$ is a uniform unparameterized
  quasi-geodesic in ${\cal C\cal G}(S)$ \cite{MM99}. 

  Let $\gamma\subset {\cal T}(S)$ be the axis of
  a pseudo-Anosov element $\phi$ and assume that
  the cotangent line of $\gamma$ consists of 
  quadratic differentials in the principal stratum. Then the
  horizontal and
  vertical measured geodesic laminations are minimal and complete
as there can not be any horizontal or vertical saddle connections
on a Teichm\"uller geodesic which projects to a closed curve in 
moduli space.

  By Proposition \ref{ccgraphqtree} and 
  Proposition \ref{infinitediam}, the image under $\Upsilon_0$
  of the line $\gamma$ is an unparameterized quasi-geodesic in
  ${\cal P\cal C}(S)$ of infinite diameter. As $\phi$ acts on this
  quasi-geodesic as a nontrivial translation,
  this quasi-geodesic is a quasi-axis for $\phi$ acting on
  ${\cal P\cal C}(S)$ and hence
  $\phi$ acts as a hyperbolic isometric on ${\cal P\cal C}(S)$.

  Vice versa, if the axis of $\phi$ is not contained in the
  principal stratum, then the support of the horizontal and
  vertical measured
  geodesic laminations defined by this axis is not complete. By
 Proposition \ref{gromovboundary}, this implies that
  the image of the axis under the map $\Upsilon_0$ has finite
  diameter in ${\cal P\cal C}(S)$. The corollary follows.
\end{proof}

\begin{remark}\label{intermediate}
  The construction of the principal curve graph, which can be thought of as an
  electrification of the usual curve graph, can be used to find a system
  of intermediate graphs which are obtained from the curve graph by adding edges and
  are obtained from the principal curve graph by deleting some edges.
  An example is the graph whose vertices are simple closed curves and where
  two such vertices $c,d$ are connected by an edge if there is at least one
  component of $S\setminus \{c,d\}$ which either is not simply connected or which
  is a polygon with at least 10 sides or if there are at least two components
  which are different from quadrangles or six-gons. It can be shown as in 
  Proposition \ref{ccgraphqtree} that this graph ${\cal E}$ is hyperbolic, and geodesics
  in ${\cal C\cal G}(S)$ map by the vertex inclusion 
  to uniform unparameterized quasi-geodesics in
  ${\cal E}$. It can also fairly easily be seen that the graph ${\cal E}$ is not a quasi-tree
  but admits a quasi-isometric embedding into a product of quasi-trees. 
  As applications of this construction are not immediate, we postpone this discussion to 
  forthcoming work. 
  \end{remark}

\section{Random walks}\label{random}

We use Proposition \ref{gromovboundary} and Corollary \ref{hyperbolic}
to show Theorem \ref{poisson} from the introduction.
For its formulation, call a subgroup $\Gamma$ of 
${\cal M\cal C\cal G}$ \emph{non-elementary} if 
it contains at least two independent pseudo-Anosov elements. 

The \emph{entropy} of a probability measure $\mu$ on
the group ${\cal M\cal C\cal G}$ is defined as 
\[H(\mu)=-\sum_{g\in {\cal M\cal C\cal G}} \mu(g)\log \mu(g).\]
  The measure $\mu$ is said
  to have \emph{finite logarithmic moment} for the action of
  ${\cal M\cal C\cal G}$ on the curve graph
  ${\cal C\cal G}(S)$ if 
  \[\sum_{g\in {\cal M\cal C\cal G}}\mu(g)
   (\max\{ 0,d_{{\cal C\cal G}(S)}(c,gc)\})<\infty.\]

The following is a more precise version of Theorem \ref{poisson}. 

\begin{theo}\label{thm:random}
  Let $\Omega$ be a random walk on ${\cal M\cal C\cal G}$ generated
  by a probability measure $\mu$ on ${\cal M\cal C\cal G}$ 
 with the following properties.
 \begin{enumerate}
 \item The support of $\mu$ generates a non-elementary semi-subgroup 
   $H$ of ${\cal M\cal C\cal G}$ 
   which contains at least one
pseudo-Anosov element whose attracting fixed point is a minimal complete
geodesic lamination. 
\item $\mu$ has finite entropy and finite logarithmic moment
for the action of ${\cal M\cal C\cal G}$ on the curve graph.
\end{enumerate}
  Then the Poisson boundary of the walk
  can be realized as a ${\cal M\cal C\cal G}$-invariant measure class
  on the space $\partial {\cal P\cal C}(S)$ of
  minimal complete geodesic laminations.
\end{theo}  
\begin{proof} Since by assumption and Corollary \ref{hyperbolic}, 
  the semigroup $H$ generated by
  the support of $\mu$ is non-elementary and contains
  at least one pseudo-Anosov element $\phi$ which acts as a hyperbolic
  isometry on ${\cal P\cal C}(S)$, the semigroup $H$ is non-elementary
  as a group of isometries acting on ${\cal P\cal C}(S)$ (apply
  Proposition \ref{gromovboundary} and Corollary \ref{hyperbolic}
  to conjugates of $\phi$).

   Theorem 1 of \cite{MT18} now shows that for any vertex
  $x\in {\cal P\cal C}(S)$, almost every sample path
  $(\omega_nx)\subset {\cal P\cal C}(S)$ converges to a point
  $\omega_+\in \partial {\cal P\cal C}(S)$. The resulting hitting
  measure $\nu$ is non-atomic, and it is the unique $\mu$-stationary
  measure on $\partial{\cal P\cal C}(S)$. 

  The same construction also applies for the action of the random
  walk on ${\cal C\cal G}(S)$. As
  $\partial {\cal P\cal C}(S)\subset \partial {\cal C\cal G}(S)$
  by Proposition \ref{gromovboundary},
  the $\mu$-stationary measure $\nu$ on $\partial {\cal P\cal C}(S)$ also
  can be viewed as a $\mu$-stationary measure on
  $\partial {\cal C\cal G}(S)$. By uniqueness, $\nu$ equals the
  hitting measure of the random walk on
  ${\cal C\cal G}(S)$.

  Now the action of ${\cal M\cal C\cal G}$ on ${\cal C\cal G}(S)$ is
  acylindrical \cite{Bw08} and therefore by Theorem 1.5 of \cite{MT18},
  the Poisson boundary of $({\cal M\cal C\cal G},\mu)$ equals
  the Gromov boundary $\partial {\cal C\cal G}(S)$ 
  of ${\cal C\cal G}(S)$,
  equipped with the hitting measure $\nu$. But this hitting measure is
  just  the unique stationary
  measure for the action of ${\cal M\cal C\cal G}$ on the
  boundary of the principal curve graph, which completes the proof of the
  theorem.
\end{proof}

\bigskip

\noindent
MATHEMATISCHES INSTITUT DER UNIVERSIT\"AT BONN\\
ENDENICHER ALLEE 60, D-53115 BONN, GERMANY


\begin{thebibliography}{CEG87}

























\bibitem[Bw08]{Bw08} B.~Bowditch, {\em Tight
    geodesics in the curve complex}, Invent. Math. 171 (2008),
  281--300.



































\bibitem[GM18]{GM18} V. Gadre, J. Maher,
  {\em The stratum of random mapping classes},
  Erg. Th. \& Dyn. Sys. 38 (2018), 2666--2682.



\bibitem[H06]{H06} U.~Hamenst\"adt, {\em Train tracks
and the Gromov boundary of the complex of curves}, 
in ``Spaces of Kleinian groups'' (Y.~Minsky, M.~Sakuma,
C.~Series, eds.), London Math. Soc. Lec. Notes 329 (2006),
187--207.

\bibitem[H09]{H09} U.~Hamenst\"adt, {\em Geometry of the
mapping class groups I: Boundary amenability}, 
Invent. Math. 175 (2009), 545--609.













\bibitem[H24]{H24} U.~Hamenst\"adt, {\em 
Periodic orbits in the thin part of strata}, J. Reine Angew. Math. 809 (2024), 41--89.





\bibitem[KM96]{KM96} V. Kaimanovich, H. Masur, {\em The Poisson boundary
    of the mapping class group}, Invent. Math. 125 (1996),  221--264.


\bibitem[KR14]{KR14} I.~Kapovich, K.~Rafi, 
{\em On hyperbolicity of free splitting and free factor complexes},
Groups, Geom. Dyn. 8 (2014), 391--414.




\bibitem[K99]{K99} E.~Klarreich, {\em The boundary at infinity
    of the curve complex and the relative Teichm\"uller space},
  unpublished manuscript. Ann Arbor (1999).

  







\bibitem[MT18]{MT18} 
J. Maher, G. Tiozzo, {\em Random walks on weakly hyperbolic
groups}, J. Reine Angew. Math. 742 (2018), 187--239.

\bibitem[Man05]{Man05} J.~Manning, {\em Geometry of pseudocharacters}, 
Geom. Topol. 9 (2005), 1147--1185. 



\bibitem[MM99]{MM99} H.~Masur, Y.~Minsky, {\em Geometry of the
complex of curves I: Hyperbolicity}, Invent. Math. 138 (1999),
103-149.


\bibitem[MM04]{MM04} H. Masur, Y. Minsky,
{\em Quasi-convexity in the curve complex}, Cont. Math. 355, 
Amer. Math. Soc., Providence, (2004), 309--320. 












\bibitem[PH92]{PH92} R.~Penner with J.~Harer, {\sl Combinatorics
of train tracks}, Ann. Math. Studies 125, Princeton University
Press, Princeton 1992.












\end{thebibliography}
\end{document}